\documentclass[11pt]{amsart}

\usepackage{epsfig,amsmath,amsfonts,latexsym}
\usepackage{color}
\usepackage{overpic}
 \usepackage{palatino}
 \usepackage{hyperref}
 \usepackage[nocompress]{cite}
 \usepackage{tikz}
\usetikzlibrary{calc,decorations.pathmorphing,shapes}

\newcounter{sarrow}
\newcommand\xrsquigarrow[1]{
\stepcounter{sarrow}
\mathrel{\begin{tikzpicture}[baseline= {( $ (current bounding box.south) + (0,-0.5ex) $ )}]
\node[inner sep=.5ex] (\thesarrow) {$\scriptstyle #1$};
\path[draw,<-,decorate,
  decoration={zigzag,amplitude=0.7pt,segment length=1.2mm,pre=lineto,pre length=4pt}] 
    (\thesarrow.south east) -- (\thesarrow.south west);
\end{tikzpicture}}
}

\newtheorem{theorem}{Theorem}

\theoremstyle{definition}

\theoremstyle{remark}
\newtheorem{remark}[theorem]{Remark}

\def\Z{\mathbb{Z}}

\newcommand{\HFred}{{{HF}}^{\mathrm{red}}}
\newcommand{\HFplus}{{{HF}}^+}

\theoremstyle{plain}

\begin{document}

\title{Mazur-type manifolds with $L$--space boundary}

\author{James Conway}

\author{B\"{u}lent Tosun}

\address{Department of Mathematics \\ University of California, Berkeley \\ Berkeley \\ California}

\email{conway@berkeley.edu}

\address{Department of Mathematics\\ University of Alabama\\Tuscaloosa\\Alabama}

\email{btosun@ua.edu}


\begin{abstract}
In this note, we prove that if the boundary of a Mazur-type $4$--manifold is an irreducible Heegaard Floer homology $L$--space, then the manifold must be the $4$--ball, and the boundary must be the $3$--sphere. We use this to give a new proof of Gabai's Property R.
\end{abstract}

\maketitle

\section{Introduction}

A \emph{Mazur-type} manifold is a contractible $4$--manifold with a particular handle structure: namely, it consist of a single handle of each index $0$, $1$, and $2$, where the $2$--handle is attached along a knot $K$ that intersect the co-core of the $1$--handle algebraically once (this yields a trivial fundamental group). Let $M(n)$ denote such a manifold, where $n\in \mathbb{Z}$ denotes the framing of the knot along which the $2$--handle is attached. Our main result is that

\begin{theorem}\label{main}
If $M(n)$ is a Mazur-type manifold, and the boundary is an irreducible Heegaard Floer homology $L$--space, then $M(n)$ is diffeomorphic to $B^4$ and $\partial M(n)$ is diffeomorphic to $S^3$.
\end{theorem}

\begin{figure}[htbp]
\begin{center}
\begin{overpic}[scale=1,tics=5]{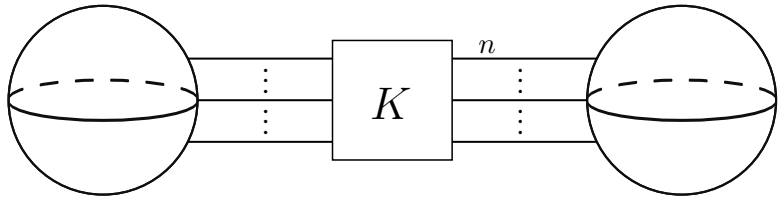}
\put(47,10){\LARGE $K$}
\put(33,13.3){$\vdots$}
\put(33,8){$\vdots$}
\put(66,13.3){$\vdots$}
\put(66,8){$\vdots$}
\put(61,18.5){$n$}
\end{overpic}
\caption{A Mazur-type manifold, with one $0$--handle, one $1$--handle, and one $2$--handle attached along $K$ with framing $n$.}
\label{mazur-manifold}
\end{center}
\end{figure}

It will follow from the proof that the attaching sphere of the $2$--handle of $M(n)$ is smoothly isotopic to $S^1 \times \{*\} \subset S^1 \times S^2$.

Recall that a \emph{Heegaard Floer homology $L$--space} (or simply \emph{$L$--space}) is a $3$--manifold whose Heegaard Floer homology is as simple as possible: $\HFred(M, \mathfrak s)$ vanishes for every $\mathfrak s \in \mathrm{Spin}^c(M)$.

\begin{remark}
Our result above provides further evidence to support Ozsv\'{a}th and Szab\'{o}'s conjecture in \cite[page 40]{OS:clay} that the full list of irreducible homology spheres that are $L$--spaces up to diffeomorphism is $S^3$ and the Poincar\'{e} homology sphere $\Sigma(2,3,5)$ with its two orientations. Theorem~\ref{main} shows that Ozsv\'{a}th and Szab\'{o}'s conjecture holds true for the class of three manifolds --- which is fairly large \cite{Akbulut:cork, AY, CH} --- that bound a contractible manifold of Mazur type.
\end{remark}

Given a handle decomposition of a Mazur-type manifold $W$, we can turn it upside down and consider it as being composed of a single handle of indices $2$, $3$, and $4$.  Attaching just the $2$--handle, we see that we have a surgery on $-\partial W$ that results in $S^1 \times S^2$.  We use this to give another proof of (a slightly more general version of) Property R, first proved by Gabai \cite{Gabai}.

\begin{theorem}\label{property r}
If $Y$ is an irreducible integer homology sphere $L$--space, and $0$--surgery on $K \subset Y$ gives $S^1 \times S^2$, then $Y$ is $S^3$ and $K$ is the unknot.
\end{theorem}

We note that our proof via Heegaard Floer homology and contact geometry is of a different flavor than Gabai's original proof, although some of the machinery in the background is similar to the machinery involved in existing proofs (by Gabai \cite{Gabai}, Gordon and Luecke \cite{GL:knotcomplement}, and Scharlemann \cite{Scharlemann}).  Our methods do not require assuming that $Y$ is $S^3$ to start off; however, the other proofs actually prove much more general results.

\begin{remark} \label{S4PC}
Recall that a well-known equivalent phrasing of the smooth $4$--dimensional Poincar\'{e} Conjecture is that every contractible manifold with boundary $S^3$ is diffeomorphic to $B^4$ (see \cite[Remark~4.8]{Yasui} and related discussion after Question~1.2 in \cite{MT:pseudoconvex}). Theorem~\ref{main} touches on this, in that it shows that whenever $S^3$ bounds a contractible manifold $M$ of Mazur-type, then $M$ is diffeomorphic to $B^4$. However, our methods do not generalize to the case of contractible manifolds with more than a single handle of index $1$ and $2$: in particular, we rely on a result \cite[Proposition~1.2]{AKa} of Akbulut and Karakurt about Mazur-type manifolds (see below for more details on this result and its proof), and its natural generalization to the more general setting is no longer true.  Indeed, if the proof of Theorem~\ref{main} generalized, then the boundary of the co-core of each $2$--handle would have to be an unknot. However, there are examples where this is not the case, see for example \cite[Section~6]{GST}.
\end{remark}

\medskip
{\bf Acknowledgements:} We thank both Jeffrey Meier and Alexander Zupan for pointing us toward Property R, and the former also for pointing us to the examples in Remark~\ref{S4PC}.  We also thank Ian Agol, John Etnyre, and Tom Mark for helpful comments. The first author was partially supported by NSF grant DMS-1344991.

\section{Proofs of Results}

\begin{proof}[Proof of Theorem~\ref{main}]
We split our proof into two steps: we first show that the boundary is $S^3$, and then we show that the $4$--manifold itself is $B^4$.

{\bf Step 1:} Assume that $Y = \partial M(n)$ is an $L$--space for some $n$. We want to show that $Y$ is diffeomorphic to $S^3$. Let $K'$ denote a meridian of $K$ (see Figure~\ref{meridian}). Thought of as a knot in $Y$, $K'$ is isotopic to the boundary of the co-core of the $2$--handle. Note that $\pm 1$--surgery on $K' \subset Y$ is $\partial M(n\mp 1)$.

\begin{figure}[htbp]
\begin{center}
\begin{overpic}[scale=1.2,tics=5]{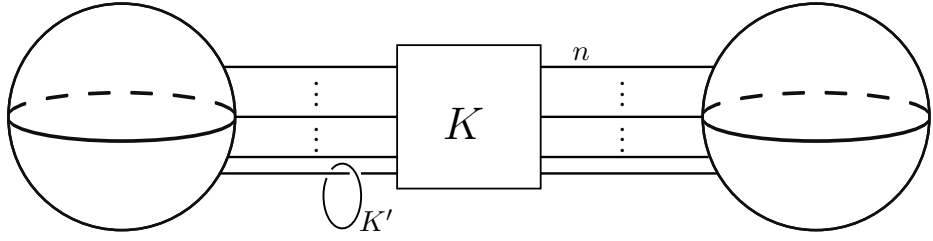}
\put(47,10){\LARGE $K$}
\put(33,13.5){$\vdots$}
\put(33,8.6){$\vdots$}
\put(66,13.5){$\vdots$}
\put(66,8.6){$\vdots$}
\put(61,18.5){$n$}
\put(38,0){$K'$}
\end{overpic}
\caption{The knot $K' \subset Y$.}
\label{meridian}
\end{center}
\end{figure}

We start by recalling that the Heegaard Floer homology of $\partial M(n)$ is independent of the framing $n$, up to a grading shift. As we mentioned in the introduction, this was proved by Akbulut and Karakurt in \cite[Proposition~1.2]{AKa}. The idea is as follows: since $M(n)$ is contractible, its boundary is an integral homology sphere with Heegaard Floer correction term $d = 0$, and hence
\[ \HFplus(\partial M(n)) \cong \mathcal{T}_{(0)}^+\oplus \HFred(\partial M(n)), \]
where $\mathcal T_{(k)}^+ \cong \mathbb F[U, U^{-1}]/\left(U\cdot \mathbb F[U]\right)$, with minimal element of grading $k$. Therefore, they just need to show that $\HFred(\partial M(n))$ is independent of $n$. This is achieved by applying the Heegaard Floer surgery exact triangle. Namely, $-1$-- and $0$--surgeries along the knot $K'$ produces $M(n+1)$ and $S^1\times S^2$, respectively, and this fits into the following surgery exact triangle:
\[ \cdots \xrightarrow{f_3} \HFplus_k(\partial M(n+1)) \xrightarrow{f_1} \HFplus_{k-\frac{1}{2}}(S^1\times S^2, \mathfrak{t}_{0}) \xrightarrow{f_2} \HFplus_{k-1}(\partial M(n)) \xrightarrow{f_3} \cdots \]
Here, $\HFplus(S^1\times S^2, \mathfrak{t}_{0})\cong \mathcal{T}^+_{(\frac{1}{2})}\oplus \mathcal{T}^+_{(-\frac{1}{2})}$, and the homomorphisms $f_1$ and $f_2$ are homogenous of degree $-\frac{1}{2}$. Using these facts, one can quickly determine that $f_3$ induces an isomorphism between $\HFred(\partial M(n))$ and $\HFred(\partial M(n+1))$ (see \cite[Proposition~1.2]{AKa} for more details). Applying Akbulut and Karakurt's result shows that if $\partial M(n)$ is an $L$--space for one value of $n$, then it is an $L$--space for all values of $n$. In particular, we know that $\pm 1$--surgery on $K'$ is an $L$--space.

We claim that the complement of $K'$ in $Y$ is irreducible. This can be seen as follows: if $Y$ is $S^3$, then our claim is true. If not, then since $Y$ is itself irreducible, if $K'$ has reducible complement, then it must be contained in a $3$--ball. If this is the case, then the result of $0$--surgery on $K'$ would be the connected sum $Y \# Y'$, for some $3$--manifold $Y'$.  However, we know that the result of $0$--surgery on $K'$ is actually $S^1 \times S^2$, since $K'$ is the meridian of $K \subset S^1 \times S^2$.  Since $S^1 \times S^2$ is prime, it follows that $K'$ must have an irreducible complement.

Since $K'$ is an $L$--space knot with irreducible complement, then by \cite[Theorem~6.5]{BBCW} (see also \cite[Page 1, paragraph 2]{LWa}), it follows that $K'$ must be fibered (this was originally proved for knots in $S^3$ in \cite{Ni, Ghiggini:fibred} and reproved in \cite{Juhasz}). On the other hand, by\cite[Corollary~1.4]{Conway:admissible}, fibered $L$--space knots support tight contact structures (originally proved for knots in $S^3$ in \cite{Hedden:positivity}).  This is proved by calculating the Heegaard Floer contact invariant of a certain contact structure on $-\left(Y_n(K')\right)$, where $n \in \Z$ is large. If $K'$ supports a contact structure with vanishing Heegaard Floer contact invariant, then one shows that the reduced Heegaard Floer contact invariant for the contact structure $-\left(Y_n(K')\right)$ is non-vanishing, which cannot happen if it is an $L$--space.

However, both $K'$ in $Y$ and $-K'$ in $-Y$ (its mirror) are fibered $L$--space knots, since both $1$-- and $-1$--surgery on $K'$ yields an $L$--space, and so they both support tight contact structures.  Let $\phi$ be the monodromy for the open book induced by $K'$; then $\phi^{-1}$ is the monodromy induced by $-K'$.  Since both $K'$ and $-K'$ support tight contact structures, then both $\phi$ and $\phi^{-1}$ must be right-veering, by \cite[Theorem~$1.1$]{HKM:right1}. However, this implies that $\phi$ must be trivial. Since $Y$ is a homology sphere, this implies that the page of the open book is a disk, that $K'$ is the unknot, and that $Y$ is diffeomorphic to $S^3$.

\begin{remark} We thank Ian Agol for remarking to us that once we know that $K'$ is fibered, it is immediate that the result of $0$--surgery on $K'$ is a surface bundle over $S^1$.  Since this surface bundle is diffeomorphic to $S^1 \times S^2$, which only fibers in one way, we can conclude that $K'$ is a genus-$0$ fibered knot, and hence is the unknot in $S^3$. \end{remark}

{\bf Step 2:} We continue to assume that $\partial M(n)$ is an $L$--space --- and hence $S^3$ --- and we now wish to show that $M(n)$ is diffeomorphic to $B^4$. First recall that if $M(n)$ admits a Stein structure in which $\partial M(n)$ is a convex level-set of the plurisubharmonic function, then $M(n)$ is a Stein filling --- and hence a strong symplectic filling --- of the tight contact structure on $S^3$.  By a famous result of Gromov and McDuff \cite{Gromov, McDuff:rationalruled}, any minimal such strong symplectic filling is diffeomorphic to $B^4$.

Let $k$ be a positive integer, such that $M(n-k)$ admits a Stein structure.  To find such a $k$, let $L \subset (S^1 \times S^2, \xi_{\rm{std}})$ be a Legendrian realization of $K$, the attaching sphere of the $2$--handle.  We can now measure $tb(L)$ (see \cite[Section~2]{Gompf} for details and conventions), such that we can build a Stein structure on $M(tb(L)-1)$ by extending the Stein structure on $S^1 \times B^3$ over a Stein $2$--handle attached to $L$ with smooth framing $tb(L)-1$.  Now, we can choose any $k$ such that $n-k \leq tb(L) - 1$.

\begin{figure}[htbp]
\begin{center}
\begin{overpic}[scale=1,tics=5]{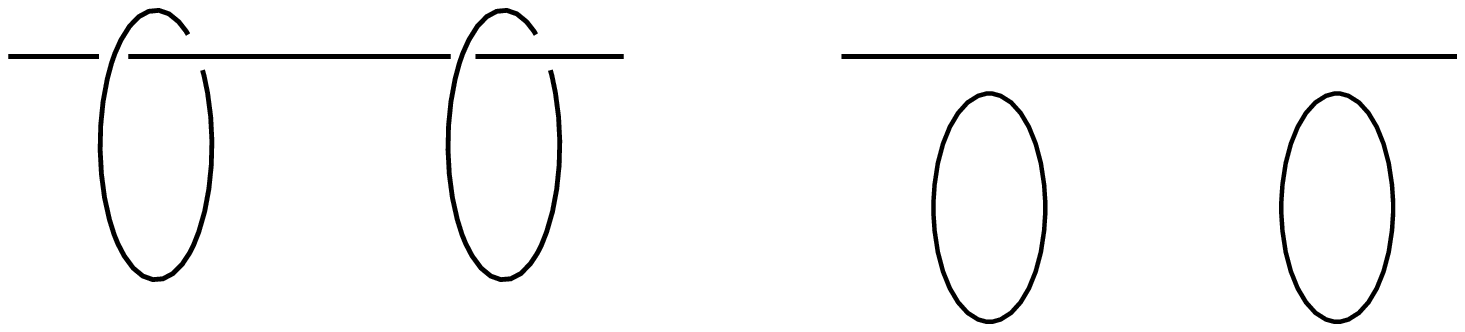}
\put(48,13){\LARGE $\xrsquigarrow{\cong}$}
\put(-3.5,17.8){$\cdots$}
\put(42.7,17.8){$\cdots$}
\put(53.8,17.8){$\cdots$}
\put(100.3,17.8){$\cdots$}
\put(18.5,19){$n-k$}
\put(77,19){$n$}
\put(19.5,10.5){\LARGE $\cdots$}
\put(14.3,8){$-1$}
\put(38.3,8){$-1$}
\put(71.8,5){$-1$}
\put(77,8){\LARGE $\cdots$}
\put(96,5){$-1$}
\end{overpic}
\caption{}
\label{blow-down}
\end{center}
\end{figure}

Since $M(n-k)$ is a Stein filling of $S^3$, we know that it is diffeomorphic to $B^4$.  Since the knot $K'$ is the unknot, the result of attaching $2$--handles with framing $-1$ along $k$ copies of $K' \subset S^3 = \partial M(n-k)$ is a $k$--fold blow-up of $B^4$ (see the left-hand side of Figure~\ref{blow-down}), which admits a symplectic structure with strongly convex boundary (see \cite[Section~7.1]{MDS:intro}).  This $4$--manifold is also diffeomorphic to the right-hand side of Figure~\ref{blow-down}, and by blowing down, we find that $M(n)$ itself admits a symplectic structure with strongly convex boundary (see again \cite[Section~7.1]{MDS:intro}).  Since $M(n)$ is minimal, the aforementioned result of Gromov and McDuff implies that $M(n)$ is diffeomorphic to $B^4$.
\end{proof}

\begin{proof}[Proof of Theorem~\ref{property r}]
Let $Y$ be an irreducible integer homology sphere $L$--space, and let $K' \subset Y$ be a knot such that $0$--surgery on $K'$ gives $S^1 \times S^2$.  Consider the $4$--dimensional cobordism from $Y$ to $S^1 \times S^2$ that is the trace of this surgery. Turn this cobordism upside down, to see it as a cobordism from $S^1 \times S^2$ to $-Y$, and glue on $S^1 \times B^3$ by a diffeomorphism $S^1 \times S^2 \cong \partial(S^1 \times B^3)$.  Call the resulting $4$--manifold $W$, and notice that $W$ is a Mazur-type manifold, and $K'$ is isotopic to boundary of the co-core of the $2$--handle in $W$.  By Theorem~\ref{main} and its proof, we know that $-Y \cong S^3$ (and hence $Y \cong S^3$ as well), and also that $K'$ is the unknot.
\end{proof}

\begin{remark} Given Property R, showing that any Mazur-type manifold with boundary $S^3$ is actually diffeomorphic to $B^4$ (Step 2 in our proof of Theorem~\ref{main}) is trivial: turning it upside down, it must consist of a $2$--handle attached along an unknot and a canceling $3$--handle, followed by a capping $4$--handle, which gives $B^4$.  However, the symplectic geometric proof presents an unusual take on this problem that we find interesting. \end{remark}

\bibliography{references}
\bibliographystyle{amsplain}

\end{document}